\documentclass[english, 12pt]{amsart}
\usepackage[T1]{fontenc}
\usepackage{geometry}
\usepackage[latin9]{inputenc}
\usepackage{amsthm}
\usepackage{amsmath}
\usepackage{amssymb}
\usepackage{setspace}
\usepackage{esint}
\usepackage{enumerate}
\usepackage{etoolbox}
\usepackage{url}
\usepackage{graphicx, epstopdf}
\usepackage[english]{babel}
\usepackage{xcolor, color}
\usepackage{soul}
\usepackage{tikz}
\usepackage{mathrsfs}
\usepackage{multirow}
\usepackage{hyperref}
\usetikzlibrary{automata,arrows,positioning,calc}

\numberwithin{equation}{section}

\makeatletter
\theoremstyle{plain}
\newtheorem{thm}{\protect\theoremname}[section]
\ifx\proof\undefined
\newenvironment{proof}[1][\protect\proofname]{\par
\normalfont\topsep6\p@\@plus6\p@\relax
\trivlist
\itemindent\parindent
\item[\hskip\labelsep
\scshape
#1]\ignorespaces
}{%
\endtrivlist\@endpefalse
}
\providecommand{\proofname}{Proof}
\fi
\theoremstyle{plain}
\newtheorem{lem}[thm]{\protect\lemmaname}
\theoremstyle{plain}

\theoremstyle{plain}

\theoremstyle{definition}

\theoremstyle{remark}

\theoremstyle{plain}

\theoremstyle{definition}

\numberwithin{figure}{section}

\makeatother

\usepackage{babel}
\providecommand{\conjecturename}{Conjecture}
\providecommand{\corollaryname}{Corollary}
\providecommand{\definitionname}{Definition}
\providecommand{\examplename}{Example}
\providecommand{\lemmaname}{Lemma}
\providecommand{\propositionname}{Proposition}
\providecommand{\remarkname}{Remark}
\providecommand{\theoremname}{Theorem}

\newcommand{\R}{\ensuremath{\mathbb{R}}}

\newcommand{\Z}{\ensuremath{\mathbb{Z}}}
\newcommand{\T}{\ensuremath{\mathbb{T}}}

\newcommand{\E}{\ensuremath{\mathbb{E}}}

\newcommand{\wh}{\widehat}
\newcommand{\ii}{\mathbf{i}}
\newcommand{\pp}[1]{p^{- \!\!\!\! #1}}
\newcommand{\qq}[1]{q^{- \!\!\!\! #1}}

\begin{document}


\title{Random walk on the randomly-oriented Manhattan lattice}
\author{Sean Ledger}
\address[SL]{School of Mathematics, University of Bristol and Heilbronn Institute for Mathematical Research, Bristol, BS8 1TW, UK.}
\email{sean.ledger@bristol.ac.uk}
\author{B\'{a}lint T\'{o}th }
\address[BT]{School of Mathematics, University of Bristol, Bristol, BS8 1TW, UK and R\'{e}nyi Institute, Budapest, HU.}
\email{balint.toth@bristol.ac.uk}
\author{Benedek Valk\'{o}}
\address[BV]{Department of Mathematics, University of Wisconsin -- Madison,  Madison,  WI 53706, USA} 
\email{valko@math.wisc.edu}

\maketitle

\begin{abstract}
In the randomly-oriented Manhattan lattice, every line in $\mathbb{Z}^d$ is assigned a uniform random direction. We consider the directed graph whose vertex set is $\mathbb{Z}^d$ and whose edges connect nearest neighbours, but only in the direction fixed by the line orientations. Random walk on this directed graph chooses uniformly from the $d$ legal neighbours at each step. We prove that this walk is superdiffusive in two and three dimensions. 
The model is diffusive in four and more dimensions. 
\end{abstract}

\section{Introduction, notation and results} \label{s:intro}

To define the randomly-oriented Manhattan lattice, let $\mathcal{E} = \{ \pm e_1, \pm e_2, \dots, \pm e_d \}$ be the canonical unit vectors in $\mathbb{Z}^d$ and let
\[
U_i := \{ x \in \mathbb{Z}^d : \langle x, e_i \rangle = 0\},
	\qquad i \in \{1,2,\dots,d \}.
\]
These  $(d-1)$-dimensional subspaces of $\mathbb{Z}^d$  allow us to uniquely index the lines in $\mathbb{Z}^d$ that are parallel to a canonical unit vector
as
\[
V(i, x) := 
	 \{ x + t e_i : t \in \Z \},
	\qquad \textrm{ for } i \in \{1,2,\dots,d \}, x \in U_i.
\]
Assign to each line $V(i, x)$, $x\in U_i$ the direction $e_i$ or $-e_i$  with probability $1/2$ each, independently of each other. For each $x\in \Z^d$ we denote by $\omega(i,x)$ the chosen direction corresponding to  the line $\{x+t e_i: t\in \Z\}$. Note that $\omega(i,x)=\omega(i, x-\langle x,e_i\rangle) $.

We study a continuous-time nearest neighbor random walk on $\mathbb{Z}^d$ in the random environment $\omega(i,x)$. The walker takes  steps at rate $d$, and if it is at $x \in \mathbb{Z}^d$ then its next position is chosen uniformly from the set $\{x+\omega(i,x), 1\le i\le d\}$. (See Figure \ref{Intro_Fig_ExampleEvolution}.)
%
%
%
Our main object of interest is the mean-square displacement 
\[
E(t) := \mathbb{E}[|X_t|^2], 	\qquad t \geq 0,
\]
where $X_t$ denotes the position of the walker at time $t$. Notice that the model is trivial when $d=1$.

We analyse the asymptotics of $E$ by applying the \emph{resolvent method}. The method was introduced in \cite{komorowski_olla_2002, landim_quastel_salmhoffer_yau_2004, yau2004} to give diffusivity estimates for  a tracer particle in a Gaussian drift field, respectively,  for a second class particle in the asymmetric exclusion process on $\Z$  and $\Z^2$.
 Later the method was used in \cite{tarres_toth_valko_2012, toth_valko_2012} to study the long-time behaviour of self-repelling diffusions pushed by the negative gradient of their local time and diffusion driven by the curl-field of the Gaussian Free Field in 2 dimensions.
 
 In this note we  show how the method transfers to the Manhattan lattice to prove superdiffusivity of the random walk in $d=2$ and $d=3$. 
 The method employed is very similar to that of  \cite{tarres_toth_valko_2012, toth_valko_2012}. However, this particular model has some interest and notoriety (see e.g.~\cite{guillotin_le_2007, guillotin_le_2008}) and the authors have been repeatedly requested to spell out the full proof.
 
%
Our main theorem provides bounds on the  Laplace transform of $E$
\[
\widehat{E}(\lambda) = \int_0^\infty e^{-\lambda t} E(t) dt,
	\qquad \textrm{as } \lambda \to 0.
\]

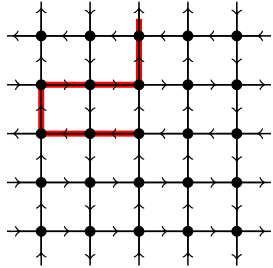
\begin{figure}
\begin{center}
\begin{tikzpicture}[scale=0.65]
\draw[-, line width = 0.8mm, color = red] (0,0) -- (-2,0) -- (-2,+1) -- (0,+1) -- (0,+2.35);

\draw[->, line width = 0.2mm] (-2,-2.5)--(-2,-2.5 + 0.08 );
\draw[->, line width = 0.2mm] (-2,-2.5)--(-2,-1.5 + 0.08);
\draw[->, line width = 0.2mm] (-2,-1.5)--(-2,-0.5 + 0.08);
\draw[->, line width = 0.2mm] (-2,-0.5)--(-2,0.5 + 0.08);
\draw[->, line width = 0.2mm] (-2,0.5)--(-2,1.5 + 0.08);
\draw[->, line width = 0.2mm] (-2,1.5)--(-2,2.5 + 0.08);

\draw[->, line width = 0.2mm] (-1,2.5)--(-1,2.5 - 0.08);
\draw[->, line width = 0.2mm] (-1,+2.5) -- (-1,1.5 - 0.08);
\draw[->, line width = 0.2mm] (-1,1.5) -- (-1,0.5 - 0.08);
\draw[->, line width = 0.2mm] (-1,0.5) -- (-1,-0.5 - 0.08);
\draw[->, line width = 0.2mm] (-1,-0.5) -- (-1,-1.5 - 0.08);
\draw[->, line width = 0.2mm] (-1,-1.5) -- (-1,-2.5 - 0.08);

\draw[->, line width = 0.2mm] (0,-2.5) -- (0,-2.5  + 0.08);
\draw[->, line width = 0.2mm] (0,-2.5) -- (0,-1.5 + 0.08);
\draw[->, line width = 0.2mm] (0,-1.5) -- (0,-0.5 + 0.08);
\draw[->, line width = 0.2mm] (0,-0.5) -- (0,0.5 + 0.08);
\draw[->, line width = 0.2mm] (0,0.5) -- (0,1.5 + 0.08);
\draw[->, line width = 0.2mm] (0,1.5) -- (0,2.5  + 0.08);

\draw[->, line width = 0.2mm] (1,-2.5) -- (1,-2.5 + 0.08);
\draw[->, line width = 0.2mm] (1,-2.5) -- (1,-1.5 + 0.08);
\draw[->, line width = 0.2mm] (1,-1.5) -- (1,-0.5 + 0.08);
\draw[->, line width = 0.2mm] (1,-0.5) -- (1,0.5 + 0.08);
\draw[->, line width = 0.2mm] (1,0.5) -- (1,1.5 + 0.08);
\draw[->, line width = 0.2mm] (1,1.5) -- (1,2.5 + 0.08);

\draw[->, line width = 0.2mm] (2,2.5)--(2,2.5 - 0.08);
\draw[->, line width = 0.2mm] (2,+2.5) -- (2,1.5 - 0.08);
\draw[->, line width = 0.2mm] (2,1.5) -- (2,0.5 - 0.08);
\draw[->, line width = 0.2mm] (2,0.5) -- (2,-0.5 - 0.08);
\draw[->, line width = 0.2mm] (2,-0.5) -- (2,-1.5 - 0.08);
\draw[->, line width = 0.2mm] (2,-1.5) -- (2,-2.5 - 0.08);

\draw[->, line width = 0.2mm] (-2.5,-2) -- (-2.5+0.08,-2);
\draw[->, line width = 0.2mm] (-2.5,-2) -- (-1.5+0.08,-2);
\draw[->, line width = 0.2mm] (-1.5,-2) -- (-0.5+0.08,-2);
\draw[->, line width = 0.2mm] (-0.5,-2) -- (0.5+0.08,-2);
\draw[->, line width = 0.2mm] (0.5,-2) -- (1.5+0.08,-2);
\draw[->, line width = 0.2mm] (1.5,-2) -- (2.5+0.08,-2);

\draw[->, line width = 0.2mm] (-2.5,-1) -- (-2.5+0.08,-1);
\draw[->, line width = 0.2mm] (-2.5,-1) -- (-1.5+0.08,-1);
\draw[->, line width = 0.2mm] (-1.5,-1) -- (-0.5+0.08,-1);
\draw[->, line width = 0.2mm] (-0.5,-1) -- (0.5+0.08,-1);
\draw[->, line width = 0.2mm] (0.5,-1) -- (1.5+0.08,-1);
\draw[->, line width = 0.2mm] (1.5,-1) -- (2.5+0.08,-1);

\draw[->, line width = 0.2mm] (2.5,0) -- (2.5-0.08,0);
\draw[->, line width = 0.2mm] (2.5,0) -- (1.5-0.08,0);
\draw[->, line width = 0.2mm] (1.5,0) -- (0.5-0.08,0);
\draw[->, line width = 0.2mm] (0.5,0) -- (-0.5-0.08,0);
\draw[->, line width = 0.2mm] (-0.5,0) -- (-1.5-0.08,0);
\draw[->, line width = 0.2mm] (-1.5,0) -- (-2.5-0.08,0);

\draw[->, line width = 0.2mm] (-2.5,+1) -- (-2.5+0.08,+1);
\draw[->, line width = 0.2mm] (-2.5,+1) -- (-1.5+0.08,+1);
\draw[->, line width = 0.2mm] (-1.5,+1) -- (-0.5+0.08,+1);
\draw[->, line width = 0.2mm] (-0.5,+1) -- (0.5+0.08,+1);
\draw[->, line width = 0.2mm] (0.5,+1) -- (1.5+0.08,+1);
\draw[->, line width = 0.2mm] (1.5,+1) -- (2.5+0.08,+1);

\draw[->, line width = 0.2mm] (2.5,2) -- (2.5-0.08,2);
\draw[->, line width = 0.2mm] (2.5,2) -- (1.5-0.08,2);
\draw[->, line width = 0.2mm] (1.5,2) -- (0.5-0.08,2);
\draw[->, line width = 0.2mm] (0.5,2) -- (-0.5-0.08,2);
\draw[->, line width = 0.2mm] (-0.5,2) -- (-1.5-0.08,2);
\draw[->, line width = 0.2mm] (-1.5,2) -- (-2.5-0.08,2);

\draw[-, line width = 0.2mm] (2.7,2) -- (-2.7,2);
\draw[-, line width = 0.2mm] (2.7,1) -- (-2.7,1);
\draw[-, line width = 0.2mm] (2.7,0) -- (-2.7,0);
\draw[-, line width = 0.2mm] (2.7,-1) -- (-2.7,-1);
\draw[-, line width = 0.2mm] (2.7,-2) -- (-2.7,-2);

\draw[-, line width = 0.2mm] (2,2.7) -- (2,-2.7);
\draw[-, line width = 0.2mm] (1,2.7) -- (1,-2.7);
\draw[-, line width = 0.2mm] (0,2.7) -- (0,-2.7);
\draw[-, line width = 0.2mm] (-1,2.7) -- (-1,-2.7);
\draw[-, line width = 0.2mm] (-2,2.7) -- (-2,-2.7);

\draw[fill] (-2,-2) circle [radius = 0.1];
\draw[fill] (-2,-1) circle [radius = 0.1];
\draw[fill] (-2,-0) circle [radius = 0.1];
\draw[fill] (-2,+1) circle [radius = 0.1];
\draw[fill] (-2,+2) circle [radius = 0.1];

\draw[fill] (-1,-2) circle [radius = 0.1];
\draw[fill] (-1,-1) circle [radius = 0.1];
\draw[fill] (-1,-0) circle [radius = 0.1];
\draw[fill] (-1,+1) circle [radius = 0.1];
\draw[fill] (-1,+2) circle [radius = 0.1];

\draw[fill] (0,-2) circle [radius = 0.1];
\draw[fill] (0,-1) circle [radius = 0.1];
\draw[fill] (0,-0) circle [radius = 0.1];
\draw[fill] (0,+1) circle [radius = 0.1];
\draw[fill] (0,+2) circle [radius = 0.1];

\draw[fill] (1,-2) circle [radius = 0.1];
\draw[fill] (1,-1) circle [radius = 0.1];
\draw[fill] (1,-0) circle [radius = 0.1];
\draw[fill] (1,+1) circle [radius = 0.1];
\draw[fill] (1,+2) circle [radius = 0.1];

\draw[fill] (2,-2) circle [radius = 0.1];
\draw[fill] (2,-1) circle [radius = 0.1];
\draw[fill] (2,-0) circle [radius = 0.1];
\draw[fill] (2,+1) circle [radius = 0.1];
\draw[fill] (2,+2) circle [radius = 0.1];
\end{tikzpicture} 
\caption{\label{Intro_Fig_ExampleEvolution} \small A typical initialisation of the random line directions in the Manhattan lattice with a legal path for the walker (red).}
\end{center}
\end{figure}

\begin{thm}
\label{Intro_Thm_MainThm}
There exists finite positive constants $C$ and  $\lambda_0$ such that for all $0<\lambda<\lambda_0$ we have
\[
\begin{array}{cc}
C^{-1} \lambda^{-9/4} \leq\widehat{E}(\lambda)\leq C \lambda^{-5/2} & \textrm{if }d=2, \text{ and}\\
C^{-1} \lambda^{-2} \log \log (\lambda^{-1}) \leq\widehat{E}(\lambda)\leq C\lambda^{-2} \log (\lambda^{-1}) \quad \ & \textrm{if }d=3.
\end{array}
\]
\end{thm}

The bounds in Theorem \ref{Intro_Thm_MainThm} are time-averaged,  they should correspond to behaviour $t^{5/4} \lesssim E(t) \lesssim t^{3/2}$ in two dimensions and $t \log \log t \lesssim E(t) \lesssim t \log t$ in three dimensions. In fact, the upper bounds on $E(t)$ can be transferred from those on $\widehat{E}(\lambda)$ by \cite[Lem.~1]{tarres_toth_valko_2012}. The lower bounds on $\widehat{E}(\lambda)$ do not transfer pointwise without some additional information on $E(t)$, but they do give the corresponding growth rates for the Ces\`{a}ro average  $\tfrac1{t} \int_0^tE(s)ds$.

In four and higher dimensions the random walk is known to be diffusive, in fact \cite{kozma_toth_2017, toth_2018} proves central limit theorem for the random walk both in the quenched and annealed environment.

The bounds on the growth rates in Theorem \ref{Intro_Thm_MainThm} are not sharp. A non-rigorous Alder--Wainwright type scaling argument (see e.g.~\cite{toth_valko_2012}) suggests that the true behavior is $E(t)\asymp  t^{4/3} $ for $d=2$, and $E(t)\asymp  t (\log t)^{1/2}$ for $d=3$. 
%
%
 A heuristic  explanation for the superdiffusive behaviour is that when the walk enters a region with many lines oriented in the same direction, it tends to follow a long and relatively straight path in that direction (see Figure \ref{Intro_Fig_10000steps}).

Random walk on the randomly-oriented Manhattan lattice is closely related to the Matheron--de Marsily (MdM) model, originally introduced in \cite{matheron_marsily_1980}. In the MdM model, one dimension has fixed uniformly chosen  random line directions as in the Manhattan lattice, but all other components are undirected. At each jump time the walker choses uniformly one of the $d$ lines going through its position. If the chosen line  is one of the directed ones  then the walker takes a step in that direction, otherwise it choses one of the two neighbours on the line randomly. 

The MdM model is well-studied \cite{guillotin_le_2007, guillotin_le_2008}, and its mean-square displacement can be analysed exactly, giving scaling $E(t) \asymp t^{3/2}$ when $d=2$ and $E(t) \asymp t\log t$ when $d=3$. There is a natural way to interpolate between the MdM and the Manhattan lattice model: suppose that  $d_\mathrm{fix} \geq 1$ of the dimensions of the lattice are directed and $d_\mathrm{free}$ are undirected, where $d_\mathrm{fix} + d_\mathrm{free} = d$. Then $d_\mathrm{fix}=1$ gives the MdM model, while $d_\mathrm{fix} =d$ gives the Manhattan lattice model. 
This offers no new models in two dimensions and all the intermediate models can be shown to be diffusive in four and higher dimension. There is, however, one potentially interesting case when $d = 3$ and $d_\mathrm{fix} = 2$. 
In Section \ref{Sect_dfixis2} we show that for this model  we have 
\begin{align}\label{MdM12}
C^{-1} \lambda^{-2} \sqrt{\log \lambda^{-1}}\le \widehat{E}(\lambda)\le C \lambda^{-2} \log(\lambda^{-1}),
\end{align}
provided that  $\lambda$ is small enough.  Non-rigorous scaling arguments suggests that the true growth is $E(t) \asymp t(\log t)^{2/3}$ in this case.

%
%
%
%
%

\begin{figure}
\begin{center}
\includegraphics[scale=0.5]{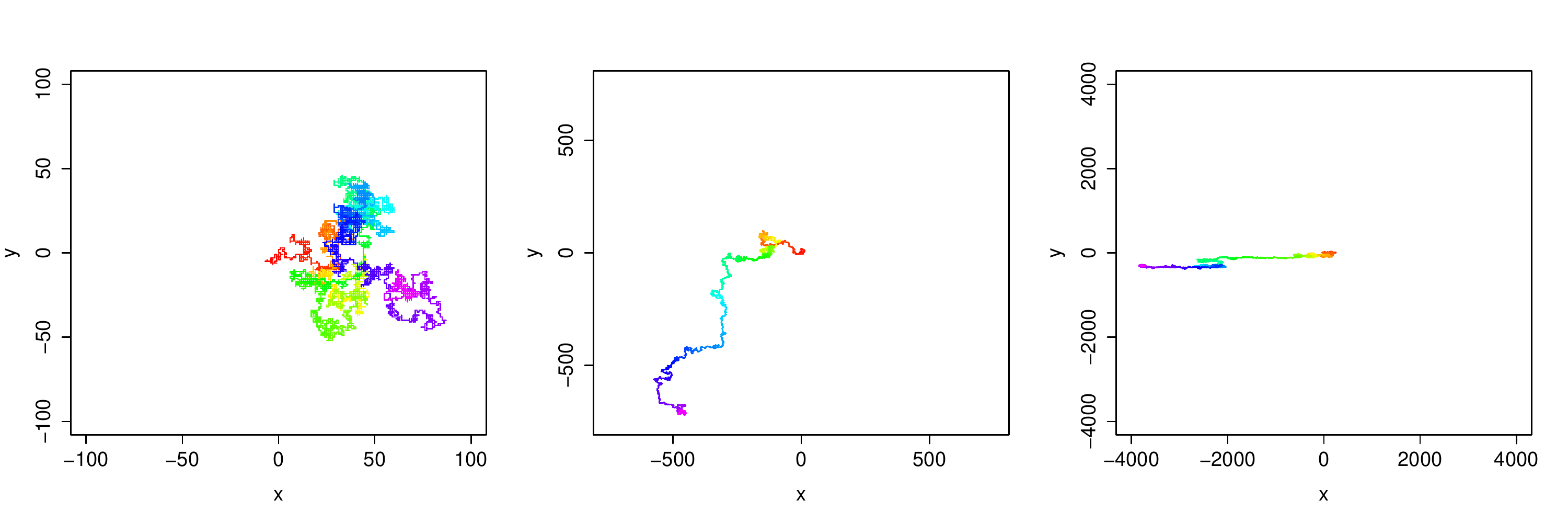}
\vspace{-0.3cm}
\caption{ \label{Intro_Fig_10000steps} \small Realisation for the first 10,000 steps: the left shows simple random walk on $\mathbb{Z}^2$, the centre shows random walk on a randomly-oriented Manhattan lattice (orientations not shown) and the right shows the MdM model with orientations in the $x$-component. Time is indexed by colour. }
\end{center}
\end{figure}

\subsection*{Notation and outline of the proof}

We will proceed by analysing the environment of randomly oriented lines as seen from the position of the random walker. Denote the set of possible environments by  
\[
\Omega = \bigotimes_{i=1}^d \bigotimes_{x \in U_i}  \{ -1,1 \}.
\]
For a given $x\in U_i$ the set $\{-1,1\}$ corresponds to $\{-e_i,e_i\}$. For an $\omega\in \Omega$ we denote the coordinates (with a slight abuse of notation) by $\omega(i,x)$, $1\le i\le d$, $x\in U_i$. 

Let 
$\tau_i : \Omega \to \Omega$  be the translation of the environment by $e_i$ and $\tau_i^{-1}$ its inverse. These act on the coordinates as 
\begin{align*}
&\hskip100pt \tau_i \omega(i,x)=\tau_i^{-1} \omega(i,x)=\omega(i,x) \qquad \text{and}\\[5pt]
&\tau_i \omega(j, x)
	= \omega(j, x +e_i), \qquad \text{and}\qquad
\tau_i^{-1} \omega(j, x)
	= \omega(j, x - e_i), \quad \text{for $j\neq i$}.
\end{align*}
The joint distribution of the environment is given by  the product measure
\begin{equation}
\label{eq:Intro_pi}
\pi= \bigotimes_{i=1}^d \bigotimes_{x \in U_i} \mu_{i,x},
\end{equation}
where $\mu_{i,x}$ is the uniform measure on $\{-1,1\}$. Note that $\pi$ is  invariant with respect to the translations $\tau_i$, $\tau_i^{-1}$. For functions $f,g \in L^2(\Omega,\pi)$ we will use the notation $(f,g)_\pi=\int f g d\pi$.

Let $\eta_t\in \Omega$ be the environment seen from the position of the random walker at time $t$. Then $\eta_t$ is a Markov process, and its generator can be expressed as follows: 
\begin{align}\label{GG}
Gf(\omega) :=  \sum_{i=1}^d \left(  
	\frac{1+\omega(i,0)}{2} f(\tau_i \omega) 
		+ \frac{1-\omega(i,0)}{2}   f(\tau_i^{-1} \omega)
		-f(\omega)
\right).
\end{align}
Note that if $\eta = (\eta_t)_{t \geq 0}$ is the environment process viewed from the walker and $X = (X_t)_{t \geq 0}$ is the  position of the walker, then 
\begin{equation}
\label{eq:Intro_InTermsOfInit}
\eta_t(i, x) = \eta_0(i, x - X_t).
\end{equation}
An important observation is that $\pi$ is an invariant measure for $\eta_t$.
It can be checked that the adjoint of $G$ is given by 
%
%
\[
G^* f(\omega) = \sum_{i=1}^d \left(  
	\frac{1-\omega(i,0)}{2} f(\tau_i \omega) 
		+ \frac{1+\omega(i,0)}{2}   f(\tau_i^{-1} \omega)
		-f(\omega)
\right),
\]
and hence the symmetric and antisymmetric parts of $G$ are given by 
\begin{align}\label{Sgen}
S& = \tfrac{1}{2}(G + G^*) \qquad\qquad Sf(\omega) =\tfrac12  \sum_{i=1}^{d} ( 
 f(\tau_i \omega) +  f(\tau_i^{-1} \omega) -2 f(\omega)
), \\
A&= \tfrac{1}{2}(G - G^*) \qquad\qquad Af(\omega) =  \sum_{i=1}^{d} \omega(i,0) (   f(\tau_i \omega) - f(\tau_i^{-1} \omega)).\label{Agen}
\end{align}
Notice that $S$ is the generator of a the environment process as seen from a symmetric simple random walk on $\mathbb{Z}^d$. 

We now sketch the basic strategy of the resolvent method. 
By symmetry we have $E(t)=\E|X_t|^2=d\cdot \E|X^1_t|^2$ where $X^1_t$ is the first coordinate of $X_t$. Observe that  $(X_t^1,\eta_t)$ is also a Markov process with a generator 
\[
\widetilde G_1 f(z, \omega)=\frac{1+\omega(1,0)}{2} f(z+1, \tau_1 \omega) 
		+ \frac{1-\omega(1,0)}{2}   f(z-1,\tau_1^{-1} \omega)
		-f(z, \omega),
\]
and that with $f(z,\omega)=z$ we get $\widetilde G_1 f(z,\omega)=\omega(1,0)$. From this it follows 
that 
\[
X_t^1=M_t+\int_0^t \phi(\eta_s) ds,
\]
where $\phi(\omega)=\omega(1,0)$, and $M_t$ is a martingale with $\E[M_t^2]=c t$ with a constant $c$.   Introduce the quantity 
\[
	E_G(t) := \mathbb{E} \Big[ \Big( \int^t_0 \phi(\eta_s) ds \Big)^2 \Big]	
		= 2\int^t_0 (t-s)\mathbb{E}[\phi(\eta_0)\phi(\eta_s)] ds.
\]
From the  inequality $\tfrac{1}{2}a^2 - b^2 \leq (a+b)^2 \leq 2a^2 + 2b^2$ and the fact that $\E[M_t^2]=ct$ it follows that   superlinear upper and lower bounds on $E_G(t)^2$ (or the corresponding bounds on its Laplace transform)  imply similar bounds (with different constants) on $\E[|X^1_t|^2]$, and hence $\E[|X_t|^2]$ as well. (In fact, it can be shown that $\E[|X_t^1|^2]=\E[M_t^2]+E_G(t)$, but this is not needed for us.)  
%
%
%
%
This means that if we give upper and lower bounds on $\widehat{E}_G(\lambda)$ that grow faster than $\lambda^{-2}$ as $\lambda\to 0$,  then these bounds will hold for $\widehat{E}(\lambda)$ as well, up to constants. Hence it is enough to estimate $\wh E_G$. 

From the definition of $E_G$ it follows that $\widehat{E}_G(\lambda) = 2\lambda^{-2} (\phi, (\lambda - G)^{-1} \phi)_\pi $. The resolvent method relies on the following variational representation of $(\phi, (\lambda - G)^{-1} \phi)_\pi $:
\begin{align}
\label{eq:VarForm_VarForm}
(\phi, (\lambda - G)^{-1} \phi)_\pi 
	&= \sup_{\psi \in L^2(\Omega, \pi)} \Big\{ 
		2(\phi, \psi)_\pi - (\psi, (\lambda - S) \psi)_\pi - (A\psi, (\lambda - S)^{-1} A\psi)_\pi
	\Big\}. 
\end{align}
(A derivation of this formula can be found in \cite{sethuraman}.) Since the right hand side of (\ref{eq:VarForm_VarForm})  is a supremum, evaluating the expression $2(\phi, \psi)_\pi - (\psi, (\lambda - S) \psi)_\pi - (A\psi, (\lambda - S)^{-1} A\psi)_\pi$ for a given $\psi \in L^2(\Omega, \pi)$ will give a lower bound on $(\phi, (\lambda - G)^{-1} \phi)_\pi$, and hence on $\widehat E(\lambda)$. The  lower bounds in Theorem \ref{Intro_Thm_MainThm} will follow from   careful choices of the test function $\psi$. The detailed proof is carried out in Section \ref{Sect_Lower}. The same idea is used for the lower bound for the intermediate MdM model with $d_\mathrm{fix} = 2$ and $d_\mathrm{free} = 1$, the proof is presented in  Section \ref{Sect_dfixis2}.

The upper bounds are easier to obtain. Note that because $S$ is self-adjoint,  the term $(A\psi, (\lambda - S)^{-1} A\psi)_\pi$ is nonnegative. Dropping it from the expression inside the supremum in (\ref{eq:VarForm_VarForm}) thus gives the following upper bound:
\[
(\phi, (\lambda - G)^{-1} \phi)_\pi 
	\le  \sup_{\psi \in L^2(\Omega, \pi)} \big\{ 
		2(\phi, \psi)_\pi - (\psi, (\lambda - S) \psi)_\pi
	\big\}=(\phi, (\lambda - S)^{-1} \phi)_\pi.
\] 
Since $S$ is the generator of the environment process as seen from a symmetric simple random walk,  $(\phi, (\lambda - S)^{-1} \phi)_\pi$ can be computed directly, which leads to the upper bounds on $\wh E_G(\lambda)$. This is demonstrated  in Section \ref{Sect_Upper}.

\subsection*{Acknowledgements} BT was supported by EPSRC (UK) Established Career Fellowship EP/P003656/1 and by OTKA (HU) K-109684. BV  was partially supported by the NSF award DMS-1712551 and the Simons Foundation.

%
%
%

\section{Proof of lower bound in Theorem \ref{Intro_Thm_MainThm}} 
\label{Sect_Lower}

Our goal will be to find an appropriate test function $\psi \in L^2(\Omega, \pi)$ where  the expression inside the supremum in (\ref{eq:VarForm_VarForm}) can be evaluated, and is sufficiently large. We will look for the test function in the form 
%
%
%
\begin{equation}
\label{eq:Lower_FormOfPsi}
\psi(\omega) := \sum_{x \in U_1} u(x) \omega(1,x),
\end{equation}
where $u \in L^2(U_1)$ is an even real function that could also depend on $\lambda$. The precise form of $u$ will be stated later in this section.  

We will start with some explicit computations involving the terms in (\ref{eq:VarForm_VarForm}). 
 In the following we will use the notation
\[
\nabla^+_i f(x) := f(x+e_i) - f(x),
\qquad \nabla_i f(x) := f(x+e_i) - f(x-e_i).
\]

\begin{lem}[Preliminary calculations]
\label{Lower_Lem_SomeCalculations}
With $\psi$ defined as in (\ref{eq:Lower_FormOfPsi}) we have:
\begin{enumerate}[(i)]
\item $(\phi, \psi)_\pi = u(0)$, $(\psi, \psi)_\pi = \Vert u \Vert_2^2$,
\item $(\psi, S\psi)_\pi = -\sum_{i=2}^d \Vert \nabla^+_i u \Vert_2^2 $,
\item $A\psi(\omega) = -\sum_{i=2}^d  \sum_{x \in U_1} \omega(i, 0) \omega(1, x) \nabla_i u(x) $,
\item Suppose that for fixed $i\neq j$ we have 
\[
\zeta(\omega)= \sum_{x\in U_i} \sum_{
y\in U_j} v(x,y) \omega(i,x) \omega(j,y)
\]
where $\sum_{x,y} v(x,y)^2<\infty$. Then 
\begin{align}\notag
S \zeta(\omega)&=\frac12 \sum_{x,y} (v(x+e_j,y)+v(x-e_j,y)-2v(x,y))\omega(i,x) \omega(j,y)\\
&+\frac12 \sum_{x,y} (v(x,y+e_i)+v(x,y-e_i)-2v(x,y))\omega(i,x) \omega(j,y)\notag\\
&+\frac12\sum_{k\neq i,j} \sum_{x,y}  (v(x+e_k,y+e_k)+v(x-e_k,y-e_k)-2v(x,y))\omega(i,x) \omega(j,y)\label{S2}
\end{align}
%
%
%
%
%
\end{enumerate}
\end{lem} 

\begin{proof}
The proof of (i) follows directly from the fact that $\omega(1,x)$ are i.i.d.~mean zero and variance 1 random variables. 

To prove (ii) first note that $\psi(\tau_1 \omega)=\psi(\tau_1^{-1} \omega)=\psi(\tau)$, and thus (after rearranging the terms) we get
\[
S\psi(\omega)=\frac{1}{2}\sum_{i=2}^d \sum_{x\in U_i} (u(x+e_i)+u(x-e_i)-2u(x)) \omega(1,x). 
\]
Hence, after a simple rearrangement of the terms we get
\begin{align*}
(\psi, S\psi)_\pi&=\frac{1}{2}\sum_{i=2}^d \sum_{x\in U_i} (u(x+e_i)+u(x-e_i)-2u(x))u(x)\\
&=-\sum_{i=2}^d  \sum_{x\in U_i} (u(x+e_i)-u(x))^2=-\sum_{i=2}^d \Vert \nabla^+_i u \Vert_2^2.
\end{align*}
Both (iii) and (iv)  follow from the definitions after some algebraic manipulations and  careful book-keeping.
%
%
%
%
%
%
\end{proof}

\subsection*{Fourier representation}
 For $f : \mathbb{Z}^d \to \R$, denote its Fourier transform by
\begin{align}\label{Fourier}
\wh f(p) := \sum_{x \in \mathbb{Z}^d} e^{\mathbf{i} \, p \cdot x} f(x)
	\qquad  p \in \mathbb{T}^d,
\end{align}
with $\mathbf{i}$ being the imaginary unit and $\mathbb{T}$ the torus on $[0,2\pi)$. 
(Although we use the same notation for the Laplace transform, it will not cause any confusion.)
For an $f:U_j\to \R$ we define $\wh f(p) $ by first extending $f$  to $\Z^d$ by setting it  equal to zero outside $U_j$ and then taking the Fourier transform. This is the same as using (\ref{Fourier}), but with a summation only on $U_j$. Note that if $f:U_j\to \R$ then $\wh f(p)$ does not depend on $p_j$, the $j$th coordinate of $p$. 

The Fourier transform of a function $c: \Z^d\times \Z^d \to \R$ is defined as 
\[
\wh c (p,q)=\sum_{x,y\in \Z^d} e^{\mathbf{i} (p \cdot x+q\cdot y)}c(x,y) \qquad p,q\in \T^d.
\]
We can extend this definition to functions  of the form $c: U_i\times U_j \to \R$ as in the single variable case. 

By Parseval's formula if  $f:U_i\to \R$ is in $L^2(U_i)$ then 
\[
\| f\|^2=\frac{1}{(2\pi)^d} \int_{\T^d} |\hat f(p)|^2 dp,
\]
and similarly, if $c:U_i\times U_j\to \R$ then 
\[
\|c\|^2=\sum_{x\in U_i} \sum_{y\in U_j} c(x,y)^2=\frac{1}{(2\pi)^{2d}} \int_{\T^d}\int_{\T^d} |\hat c(p,q)|^2 dp dq. 
\]
For an $u: U_1\to \R$ and  $j\neq 1$ we have 
\[
\wh{\nabla^+_j u}(p)=(e^{-\ii p_j}-1) \wh u(p).
\]
Note that $|e^{-\ii t}-1|^2=4 \sin^2(\tfrac{t}{2})$. For $p\in \T^d$ let 
\[
\wh d(p)=\sum_{j=1}^d 4 \sin^2(\tfrac{p_j}{2}),
\]
and define $\pp{j}=p-p_j \,e_j$ as the vector obtained from $p$ by replacing its $j$th coordinate with  $0$. 
Then with  $\psi$ defined as in \eqref{eq:Lower_FormOfPsi} we have
\begin{align}\label{J2}
(\psi, (\lambda-S)\psi)_\pi=\frac{1}{(2\pi)^d} \int_{\T^d} \left(\lambda+ \wh d(\pp{1})\right) | \wh u(p)|^2 \, dp.
\end{align}
Now suppose that $\zeta$ is defined as in part (iv) of Lemma \ref{Lower_Lem_SomeCalculations}. According to the lemma, we can express $S\zeta(\omega) $ as $ \sum_{x\in U_i} \sum_{
y\in U_j} v^*(x,y) \omega(i,x) \omega(j,y)$ where  $v^*$ can be read off from (\ref{S2}). From this the Fourier transform of $v^*$ can be expressed  as follows:
\begin{align*}
\wh v^*(p,q)&=\frac{1}{2}\left\{(e^{\ii p_j}+e^{-\ii p_j}-2)+(e^{\ii q_i}+e^{-\ii q_i}-2)+ \sum_{k\neq i,j} (e^{\ii (p_k+q_k)}+e^{-\ii (p_k+q_k)}-2)\right\} \wh v(p,q)\\
&=- \frac12 \wh d(\pp{i}+\qq{j}) \wh v(p,q).
\end{align*}
This also shows that $(\lambda-S)^{-1} \zeta(\omega)$ can be expressed as $ \sum_{x\in U_i} \sum_{
y\in U_j} s(x,y) \omega(i,x) \omega(j,y)$ with 
\[
\wh s(p,q)=\left(\lambda+\tfrac12 \wh d(\pp{i}+\qq{j})\right)^{-1} \wh v(p,q). 
\]
By Lemma \ref{Lower_Lem_SomeCalculations} we can express 
\[
A\psi(\omega)=\sum_{i=2}^d \sum_{x\in U_i}  \sum_{y\in U_1} v_i(x,y)  \omega(i,x) \omega(1,y)
\]
where 
\[
v_i(x,y)=-\mathbf{1}\{x=0\}  \nabla_i u(y), \quad \text{and} \quad  \wh v_i(p,q)=(e^{\ii q_i}-e^{-\ii q_i}) \wh u(q).
\]
This leads to the identity
\begin{align}\label{AS1}
(A\psi, (\lambda-S)^{-1} A\psi)_\pi&=\frac{1}{(2\pi)^{2d}} \int_{\T^d}\int_{\T^d} 
\sum_{i=2}^d 4\sin^2({q_i}) \left(\lambda+\tfrac12 \wh d(\pp{i}+\qq{1})\right)^{-1} 
|\wh u(q)|^2
 dp dq.
\end{align}
The proceeding integral inequalities follow from simple calculus, comparing $\sin^2(x/2)$ to $x^2$ on $(-\pi,\pi)$. 
\begin{lem}\label{lem:integr}
If $d=2$ then for all $\lambda>0$ and $p\in \T^2$ we have
\begin{align}\label{D2}
\frac{1}{(2\pi)^{d}} \int_{\T^d}  \left(\lambda+\tfrac12 \wh d(\qq{2}+\pp{1})\right)^{-1} dp\le C \lambda^{-1/2},
\end{align}
where $C$ is a finite constant. 
If $d=3$ then for all $p\in \T^3$ and $0<\lambda\le 1/3$ we have
\begin{align}\label{D3}
\frac{1}{(2\pi)^{d}} \int_{\T^d}  \left(\lambda+\tfrac12 \wh d(\qq{2}+\pp{1})\right)^{-1} dq\le C |\log( \lambda+\tfrac1{2} \sin^2(\tfrac{p_2}{2}))|,
\end{align}
where $C$ is a finite constant.
\end{lem}
\subsection*{Estimating $(\phi, (\lambda - G)^{-1} \phi)_\pi $ using $\psi$}
Our goal is to give a lower bound on $2(\phi, \psi)_\pi - (\psi, (\lambda - S) \psi)_\pi - (A\psi, (\lambda - S)^{-1} A\psi)_\pi$ when $\psi$  is of the form (\ref{eq:Lower_FormOfPsi}), this will also give a lower bound for $(\phi, (\lambda - G)^{-1} \phi)_\pi $.

 By the inverse Fourier formula we have
\begin{align}\label{J1}
(\phi, \psi)_\pi =u(0)=\frac{1}{(2\pi)^d} \int_{\T^d} \wh u(p) dp.
\end{align}
We assumed that $u$ is even, hence $\wh u(p)$ is real. 

Using the expression  (\ref{AS1}) for $(A\psi, (\lambda - S)^{-1} A\psi)_\pi$ and the bounds of Lemma  \ref{lem:integr} we get that 
\begin{align}\label{J3d2}
(A\psi, (\lambda - S)^{-1} A\psi)_\pi\le \frac{C}{(2\pi)^d} \int_{\T^d} \lambda^{-1/2}  \sin^2({p_2}) |\wh u(p)|^2 dp \qquad \text{for $d=2$}
\end{align} 
and
\begin{eqnarray}\nonumber
&&(A\psi, (\lambda - S)^{-1} A\psi)_\pi
\\&&\hskip50pt\le  \frac{C}{(2\pi)^d} \int_{\T^d} \sum_{j=2}^3 |\log (\lambda+\tfrac1{2} \sin^2(\tfrac{p_j}{2}))| \sin^2({p_j})) |\wh u(p)|^2 dp\qquad \text{for $d=3$,}\label{J3d3}
\end{eqnarray}
if  $0<\lambda\le 1/3$. Now we have all the ingredients to prove the lower bounds  in Theorem \ref{Intro_Thm_MainThm}.

\begin{proof}[Proof of the lower bound for $d=2$ in Theorem \ref{Intro_Thm_MainThm}]
If $d=2$ then (\ref{J2}),  (\ref{J1}) and (\ref{J3d2})  show that for a $\psi$  of the form (\ref{eq:Lower_FormOfPsi}) we have
\begin{align*}
&2(\phi, \psi)_\pi - (\psi, (\lambda - S) \psi)_\pi - (A\psi, (\lambda - S)^{-1} A\psi)_\pi \\
&\qquad \ge 
\frac{1}{(2\pi)^2} \int_{\T^2} \left( 2\wh u(p) - \left(\lambda+ \wh d(\pp{1})\right) | \wh u(p)|^2-C \lambda^{-1/2}  \sin^2({p_2}) |\wh u(p)|^2\right)  dp,
\end{align*}
with a fixed constant $C>0$. 

The integral achieves its maximum for the choice 
\[
\wh u(p)=\frac{1}{\lambda+\wh d(\pp{1})+C\lambda^{-1/2} \sin^2({p_2})  },
\]
note that this is real, bounded and only depends on $p_2$, thus it corresponds to a function $u:U_1\to \R$ that satisfies our assumptions. The  value of the integral for this particular $u$ is 
\[
\frac{1}{2\pi} \int_{\T } \frac{1}{\lambda+4 \sin^2(\tfrac{p_2}{2}) +C\lambda^{-1/2} \sin^2({p_2})  } dp_2
\]
which can be bounded from below by $C' \lambda^{-1/4}$. This means that with this particular choice of $\psi$ the value of $2(\phi, \psi)_\pi - (\psi, (\lambda - S) \psi)_\pi - (A\psi, (\lambda - S)^{-1} A\psi)_\pi $ is at least $C' \lambda^{-1/4}$, hence $(\phi, (\lambda - G)^{-1} \phi)_\pi \ge C' \lambda^{-1/4}$. Thus $\wh E_G(\lambda)$ grows faster than $\lambda^{-9/4}$ as $\lambda\to 0$, from which the lower bound of  Theorem \ref{Intro_Thm_MainThm} on $\wh E(\lambda)$ follows. 
\end{proof}

\begin{proof}[Proof of the lower bound for $d=3$ in Theorem \ref{Intro_Thm_MainThm}]
In the $d=3$ case (\ref{J2}),  (\ref{J1}) and (\ref{J3d3}) lead to  
\begin{align*}
&2(\phi, \psi)_\pi - (\psi, (\lambda - S) \psi)_\pi - (A\psi, (\lambda - S)^{-1} A\psi)_\pi \\
&\qquad \ge 
\frac{1}{(2\pi)^3} \int_{\T^3} \left( 2\wh u(p) - \left(\lambda+ \wh d(\pp{1})+ C \sum_{j=2}^3 |\log (\lambda+\tfrac1{2} \sin^2(\tfrac{p_j}{2}))| \sin^2({p_j})\right) |\wh u(p)|^2 \right)  dp,
\end{align*}
assuming $0<\lambda\le 1/3$.
The integral takes its maximum for the choice
\[
\wh u(p)=\left(\lambda+ \wh d(\pp{1})+ C\sum_{j=2}^3 |\log (\lambda+\tfrac1{2} \sin^2(\tfrac{p_j}{2}))| \sin^2({p_j})\right)^{-1}
\]
which correspond to a valid function $u:U_1\to \R$. The value of the integral is 
\begin{align}\label{J30d3}
\frac{1}{(2\pi)^3} \int_{\T^3}\left(\lambda+ \wh d(\pp{1})+C \sum_{j=2}^3 |\log (\lambda+\tfrac1{2} \sin^2(\tfrac{p_j}{2}))| \sin^2({p_j})\right)^{-1} dp.
\end{align}
This integral is comparable (up to constants) to the integral
\[
\int_0^\pi \int_0^\pi \frac{dx\, dy}{\lambda+x^2+y^2+ x^2|\log(\lambda+ x^2)|+y^2|\log(\lambda+ y^2)|}
\]
which can be shown to be at least $C' \log \log(\lambda^{-1})$ for $0<\lambda\le 1/3$. The proof of the statement now follows as in the $d=2$ case. 
\end{proof}

\section{Proof of the upper bounds in Theorem \ref{Intro_Thm_MainThm}} 
\label{Sect_Upper}

As explained at the end of Section \ref{s:intro}, we have the bound
\[
(\phi, (\lambda - G)^{-1} \phi)_\pi 
	\le  (\phi, (\lambda - S)^{-1} \phi)_\pi.
\] 
If $\psi$ is of the form of  (\ref{eq:Lower_FormOfPsi}) then $(\lambda-S)^{-1} \psi$ can  be written as $\sum_{x \in U_1} u^*(x) \omega(1,x) $ with $\wh u^*(p)=\frac{1}{\lambda+\wh d(\pp{1})} \wh u(p)$. Since $\phi(\omega)=\sum_{x \in U_1} \mathbf{1}\{x=0\} \omega(1,x)$, we obtain
\begin{align}\label{S9}
 (\phi, (\lambda - S)^{-1} \phi)_\pi=\frac{1}{(2\pi)^d} \int_{\T^d} \frac{1}{\lambda+\wh d(\pp{1})} dp.
\end{align}
The integral in (\ref{S9}) can be bounded by $C \lambda^{-1/2}$ if $d=2$ and $C \log(\lambda^{-1})$ if $d=3$ and $0<\lambda<1/2$. From this the upper bounds in  Theorem \ref{Intro_Thm_MainThm} follow. 
Note also that for $d\ge 4$ the integral in  (\ref{S9})  can be bounded by a constant depending on $d$ and not $\lambda$, which shows that in these cases the model is diffusive.

\section{ Bounds for the MdM model with $d_\mathrm{fix} = 2$, $d_\mathrm{free} = 1$ }
\label{Sect_dfixis2}

Consider the modification of the three-dimensional MdM model 
with $d_\mathrm{fix} = 2$ and $d_\mathrm{free} = 1$, and assume that the $e_1, e_2$ directions are fixed.  Then the generator of this process is similar to (\ref{GG}), but  the $i=3$ term in the sum is replaced by $\tfrac{1}{2} f(\tau_i \omega) + \tfrac{1}{2} f(\tau_i^{-1} \omega) - f(\omega)$. Note that the symmetric part is still the same $S$ as in (\ref{Sgen}) for $d=3$, but the asymmetric part will only have the terms $i=1$ and $2$ from (\ref{Agen}). 

Because the symmetric part is the same as in the case of the $d=3$ Manhattan model, the upper bound proved there holds for this model as well.

For the lower bound we can also proceed with a similar argument as  in the case of the Manhattan model, the only modification is that  bound in  (\ref{J30d3}) now will only consist of the $j=2$ term. Hence we get
\begin{align*}
&2(\phi, \psi)_\pi - (\psi, (\lambda - S) \psi)_\pi - (A\psi, (\lambda - S)^{-1} A\psi)_\pi \\
&\qquad \ge 
\frac{1}{(2\pi)^3} \int_{\T^3} \left( 2\wh u(p) - \left(\lambda+ \wh d(\pp{1})+ C |\log (\lambda+\tfrac1{2} \sin^2(\tfrac{p_2}{2}))| \sin^2({p_2}))\right) |\wh u(p)|^2 \right)  dp,
\end{align*}
which leads to the following lower bound:
\begin{align*}
(\phi, (\lambda - G)^{-1} \phi)_\pi \ge \frac{1}{(2\pi)^3} \int_{\T^3}\left(\lambda+ \wh d(\pp{1})+C  |\log (\lambda+\tfrac1{2} \sin^2({p_2}))| \sin^2(\tfrac{p_2}{2})\right)^{-1} dp.
\end{align*}
For $0<\lambda\le 1/3$ the integral on the right is comparable to 
\[
\int_0^\pi \int_0^\pi \frac{dx\, dy}{\lambda+x^2+y^2+ x^2|\log(\lambda+ x^2)|}
\]
which can be further bounded from below by a constant times $\int_0^\pi \int_0^\pi \frac{dx\, dy}{\lambda+y^2 +x^2 \log(\lambda^{-1})}$. This integral is at least $C \sqrt{\log \lambda^{-1}}$ for $\lambda$ small, which leads to the lower bound in (\ref{MdM12}).

\bibliographystyle{alpha}

\begin{thebibliography}{99}

\bibitem{guillotin_le_2007}
N.~Guillotin-Plantard and A.~Le Ny.
\newblock Transient random walks in dimension two.
\newblock \emph{Theo.~Probab.~Appl.}, 52(4):815-826, 2007.


\bibitem{guillotin_le_2008}
N.~Guillotin-Plantard and A.~Le Ny.
\newblock A functional limit theorem for a 2d-random walk with dependent marginals.
\newblock \emph{Electronic Communications in Probability}, 13(34):337-351, 2008.

%
%

\bibitem{komorowski_olla_2002}
T.~Komorowski and S.~Olla.
\newblock On the superdiffusive behaviour of passive tracer with a Gaussian drift.
\newblock \emph{J.~Stat.~Phys.} 108:647-668, 2002.


\bibitem{kozma_toth_2017}
G.~Kozma and B.~T\'{o}th.
\newblock Central limit theorem for random walks in doubly stochastic random environment: $\mathscr{H}_{-1}$ suffices.
\newblock \emph{Annals of Probability}, 45:4307-4347, 2017. 
 
 
\bibitem{landim_quastel_salmhoffer_yau_2004}
C.~Landim, J.~Quastel, M.~Salmhoffer and H.-T.~Yau. 
\newblock Superdiffusivity of asymmetric exclusion process in dimensions one and two. 
\newblock \emph{Communications in Mathematical Physics}, 244:455-481, 2004.


\bibitem{matheron_marsily_1980}
G.~Matheron and G.~de Marsily.
\newblock Is transport in porous media always diffusive? A counterxample.
\newblock \emph{Water Resources Res.}, 16:901-907, 1980.


\bibitem{sethuraman}
S.~Sethuraman.
\newblock Central limit theorems for additive functionals of the simple exclusion process
\newblock \emph{Annals of Probability}, 28:277-302, 2000.


\bibitem{tarres_toth_valko_2012}
P.~Tarr\`{e}s, B.~T\'{o}th and B.~Valk\'{o}.
\newblock Diffusivity bounds for 1d Brownian polymers 
\newblock \emph{Annals of Probability}, 40:695-713, 2012.


\bibitem{toth_2018}
B.~T\'{o}th. 
\newblock Quenched central limit theorem for random walks in doubly stochastic random environment.
\newblock To appear: \emph{Annals of Probability}. \texttt{arXiv:1704.06072}, 2018. 


\bibitem{toth_valko_2012}
B.~T\'{o}th and B.~Valk\'{o}.
\newblock Superdiffusive bounds on self-repellent Brownian polymers and diffusion in the curl of the Gaussian free field in $d=2$.  
\newblock \emph{Journal of Statistical Physics}, 147:113-131, 2012. 

\bibitem{yau2004}
H.-T.~Yau.
\newblock $(\log t)^{2/3}$ law of the two dimensional asymmetric simple exclusion process.
\newblock \emph{Annals of Mathematics}, 159:377-405, 2004
\end{thebibliography}

\end{document}